\documentclass[12pt,a4paper]{article}

\usepackage[T1]{fontenc}
\usepackage[utf8]{inputenc}

\usepackage{textcomp}
\usepackage[official]{eurosym}
\usepackage{amssymb,mathrsfs}
\usepackage{amsthm}
\usepackage{mathtools}
\usepackage{latexsym}

\newtheorem{theorem}{Theorem}[section]
\newtheorem{lemma}[theorem]{Lemma}

\newtheorem{corollary}[theorem]{Corollary}
\newtheorem{remark}[theorem]{Remark}

\newtheorem{definition}[theorem]{Definition}

\newcommand{\2}{NONEMPTY}
\newcommand{\1}{EMPTY}
\newcommand{\kare}{$\square$}

\begin{document}
\title{An Extension of Galvin’s Theorem Related to the Banach-Mazur Game}
\author{Servet Soyarslan$^{(a)}$\footnote{Corresponding author. \newline E-mail addresses: servet.soyarslan@gmail.com (S. Soyarslan),  osul@metu.edu.tr(S. \"Onal).} , S\"uleyman \"Onal$^{(b)}$ }

\maketitle	

{\scriptsize a. Independent researcher, Ankara, Türkiye.
	
	b. Middle East Technical University, Department of Mathematics, 06531 Ankara, Türkiye.} 
	\maketitle

\begin{abstract}

Let $X$ be a topological space. Galvin's theorem on the Banach-Mazur game states that if $X$ has a $\pi$-base $\mathcal{B}$ satisfying the following condition:  

\textit{For any $U \in \mathcal{B}$, there exists a disjoint family $\mathcal{A}_U$ of nonempty open subsets of $U$ such that}  
\[
|\mathcal{A}_U| \geq |\{ B \in \mathcal{B} : U \subseteq B \}|,
\]  

then, in the Banach-Mazur game on $X$, the second player, denoted NONEMPTY, has a  2-tactic if \2 has a winning strategy in this game.  
In this paper, we extend the disjointness condition of this family by replacing it with the  condition of being open finite family.
With this extension of the family $\mathcal{A}_U$, we showed that, the theorem is also true.

\end{abstract}

	

\section{Introduction} 
There is a topological space in which \2 has a winning strategy in the Banach-Mazur game on this space, but there is no 1-tactic \cite{debs1985strategies}. In this case, it is impossible for \2 to transition from a winning strategy to a 1-tactic in this game. However, what can be said about the other tactics?
Galvin's theorem is an important result that shows the existence of \2's a 2-tactic  when \2 has a winning strategy in the Banach-Mazur game. Moreover, a set-theoretic statement related to this theorem has been formulated and studied (\cite{brian2021telgarsky}, \cite{brian2022independence}). For the proof of Galvin's theorem, see Theorem 2.4 in \cite{brian2021telgarsky}. It was shown in \cite{soyarslanNoeth} that Theorem 2.4 in \cite{brian2021telgarsky} and the paragraph at the end of the theorem, where the authors mention that the hypothesis of this theorem could be slightly weakened, are equivalent to Galvin's theorem.

The disjointness of the family $\mathcal{A}_U$ holds a crucial role in the proof of Galvin's theorem. Specifically, for the construction of the desired 2-tactic, given two open sets, the disjointness ensures that there is a unique element of $\mathcal{A}_U$ that is contained as a subset of one and contains the other as a subset. This unique element encodes the history of the play up to the current move. Consequently, the winning strategy is applied through this element using the encoded history of the play. In our main theorem (Theorem \ref{Teorem winning ten 2-tactige ana sonuc open-finite}
), we extended the notion of disjointness in this family by replacing it with the condition of being an open-finite family. However, this extension is far from straightforward, as the open-finite family does not provide a unique element capable of encoding the entire history of the play. Our result overcomes this difficulty and establishes the desired extension to open-finite families. Our main theorem not only serves as an extension of Galvin's theorem but also generalizes Theorem 3.8 in \cite{soyarslanNoeth}, which includes Galvin's theorem. In addition to the extended disjointness mentioned here, our main theorem employs the concepts of Noetherian table and purification of a $\pi$-base, which will be defined below. We believe that these concepts have the potential to possibly further extend Galvin's theorem.

   \vspace{2mm}\textbf{Notations and Terminology}\vspace{2mm}
   
   The collection of nonempty elements in the topology of a space $X$ is denoted by $\tau^*(X)$. For any $\mathcal{A} \subseteq \tau^*(X)$, $pur(\mathcal{A})=\{U \in \mathcal{A}.$ every nonempty open subset $W$ of $U$ possesses a cellular family with cardinality $\omega\}$. A base or $\pi$-base for a topological space $X$ in this work does not include the empty set.   If $A$ is a subset of a space $X$, we use $\overline{A}$ to denote the closure of $A$ in $X$. For any set $A$, we use $\mathscr{P}(A)$ and $|A|$ to represent the power set and cardinality of $A$, respectively. If a set $A$ is written as $A = \{a_i: i \in \alpha\}$, where $\alpha$ is a cardinal or ordinal number, we always assume that this correspondence is one-to-one, i.e., $a_i \neq a_j$ if $i \neq j$.

For a topological space, unless stated otherwise, we do not assume any separation axioms.

\vspace{2mm}\textbf{The Banach-Mazur Game and Strategies for \2}\vspace{2mm}

Consider a topological space $X$ and two players, denoted \1 and \2. In round 1, \1 selects $U_0 \in \tau^*(X)$, and \2 chooses a nonempty open subset $V_0 \subseteq U_0$. In subsequent rounds, \1 picks $U_n \subseteq V_{n-1}$, and \2 responds with $V_n \subseteq U_n$. The sequence $(U_0, V_0, U_1, V_1, \ldots)$ forms a play. If $\bigcap_{i \in \omega} U_i = \emptyset$, \1 wins, otherwise \2 wins. The collection of all such plays is the \textit{Banach-Mazur game on a space $X$}, $BM(X)$.

A finite sequence $(U_0, V_0, \ldots, U_n, V_n)$ of legal moves is called a \textit{play part}, and a sequence ending with $(U_n)$ is an \textit{added play part}.

A \textit{winning strategy} for \2 is a function $\sigma$ that guarantees \2’s victory, defining their moves at each round. The domain of $\sigma$ consists of added play parts, and its range is the possible moves for \2. If \2 follows $\sigma$, the resulting play is: $$(U_0, V_0, U_1, V_1, \ldots)$$ where $\sigma((U_0, V_0, \ldots, U_i)) = V_i$ for all $i$, and $\bigcap_{i \in \omega} U_i \neq \emptyset$.

A $k$-\textit{tactic} for \2 is a function that guarantees a win based on the previous $k$ moves of \1. A 1-tactic is called a \textit{stationary strategy}. (We use "$k$-\textit{tactic}" to mean "winning $k$-\textit{tactic}" for brevity and because we are not concerned with strategies that do not guarantee a win.)

 \vspace{2mm}\textbf{Noetherian Families and Family Purification}

\begin{definition}\label{o-f-f}
Let $X$ be space and $\mathcal{A}$ be a family of nonempty subsets of $X$. It is said that $\mathcal{A}$ is an open-finite family if $|\{A\in \mathcal{A}: V\subseteq A\}|<\omega$  for any $V\in \tau^*(X)$.
\end{definition}

\begin{definition}\label{Noetherian} A family  $\mathcal{A}\neq \emptyset$ is called \textbf{Noetherian} if $\emptyset\notin \mathcal{A}$ and every strictly increasing sequence  $A_1\subsetneqq A_2 \subsetneqq A_3\subsetneqq \ldots$ of elements of $\mathcal{A}$ is finite.
	
	It is easy to check that $\mathcal{A}$ is Noetherian iff for all nonempty $\mathcal{B}\subseteq \mathcal{A}$ and for all $B\in \mathcal{B}$ there exists a $\subseteq$-maximal element $B^*$ of $\mathcal{B}$ such that $B\subseteq B^*$.   \begin{flushright}
		\kare
	\end{flushright}
\end{definition}

\begin{definition} \label{Definition Rank of Noetherian families} \cite{onal2020strategies}
	Let  $\mathcal{A}$ be a Noetherian family. We take any ordinal $i$ and  define $\mathcal{A}_i$ as the family  of $\subseteq$-maximal elements of $\mathcal{A}-\bigcup_{j<i} \mathcal{A}_j$ i.e., $\mathcal{A}_i=\{A\in \mathcal{A}-\bigcup_{j<i} \mathcal{A}_j : \forall B\in \mathcal{A}-\bigcup_{j<i} \mathcal{A}_j(A\subseteq B\rightarrow B=A) \}$, where $\mathcal{A}_j$ is already defined  for all $j<i$ in the same way.  
	Say the minimum element of the class   $\{\alpha:\mathcal{A}=\bigcup_{i<\alpha} \mathcal{A}_i\}$ of ordinals is $\alpha^*$ and so,  we have $\mathcal{A}=\bigcup_{i<\alpha^*}\mathcal{A}_i$. (Note that the family $\{\mathcal{A}_i:i<\alpha^*\}$ is disjoint and  $\mathcal{A}_i\neq \emptyset$ for all $i<\alpha^*$) Then, we define the following.
	
	1) The ordinal $\alpha^*$ is called \textbf{rank of $\mathcal{A}$} and denoted by \textbf{rank$(\mathcal{A})$}.
	
	2) What we mean with \textbf{
		Noetherian union of $\mathcal{A}$} is	$\bigcup_{i<rank(\mathcal{A})}\mathcal{A}_i$.
	
	3) We define a  function, \textbf{ rank function of  $\mathcal{A}$},   $r$ from $\mathcal{A}$ to $rank(\mathcal{A})$ such that $r(A)=i$ iff $A\in  \mathcal{A}_i$ for every $A\in \mathcal{A}$. The ordinal  $r(A)$ is called \textbf{level (or rank) of A in  $\mathcal{A}$}.  \begin{flushright}
		\kare
	\end{flushright}
	
\end{definition}

\begin{definition}\cite{soyarslanNoeth}
	Let  $\mathcal{A}$ be a Noetherian family and let Noetherian union of $\mathcal{A}$ be	$\bigcup_{i<\alpha} \mathcal{A}_i$. What we mean with \textbf{an $\alpha\times \beta$ Noetherian table of $\mathcal{A}$} or shortly \textbf{a Noetherian table of $\mathcal{A}$}, which is denoted by $[\mathcal{A}]$, is the following. 
	
	\[\begin{bmatrix}
		A^0_0&A^1_0&A^2_0&\dots  \\
		A^0_1&A^1_1&A^2_1&\dots\\
		A^0_2&A^1_2&A^2_2&\dots\\
		\dots&\dots&\dots& \dots \\
		
	\end{bmatrix}\]
	
	There are $\alpha$, which is rank of $\mathcal{A}$, many rows in this Noetherian table and for each $i\in \alpha$ there  exists an ordinal $\theta_i$ such that $\mathcal{A}_i=\{A^j_i\in \mathcal{A}_i:j \in\theta_i\}$.  The  $i$-th row is the tuple $(A^0_i,A^1_i,A^2_i,\ldots)$ consisting of all elements of   $\mathcal{A}_i$.  $\beta=\min\{\xi:\xi$ is an ordinal and $\theta_i\leq  \xi$ for all $i\in \alpha \}$.

	For any $A^j_i$ in this Noetherian table, note that    $r(A^j_i)=i$, the ordinal $j$ is called  \textbf{ order of $A^j_i$} in $\mathcal{A}_i$ of $[\mathcal{A}]$ and it is written: $o(A^j_i)=j$.\begin{flushright}
		\kare
	\end{flushright}

\end{definition} 

	Every topological space has a Noetherian $\pi$-base (Corollary 2.17 in \cite{soyarslanNoeth}). In fact, we have the following result.

	\begin{theorem} (Corollary 2.14 in \cite{soyarslanNoeth})
		Let $X$ be space and $\pi w(X)\leq\kappa$. For any $\pi$-base  $\mathcal{B}$  of  $X$, there exists a Noetherian $\pi$-base $\mathcal{B}'\subseteq \mathcal{B}$ such that

		(i) $rank(\mathcal{B'})\leq \kappa$ and $\mathcal{B'}$ has a $rank(\mathcal{B'})\times \beta$  Noetherian table  where  $\beta\leq \kappa$,

		(ii) $|\{B\in \mathcal{B'}: W\subseteq B \}|<\kappa$ for all $W\in \mathcal{B'}$.	 \begin{flushright}
			\kare
		\end{flushright}
	\end{theorem}

\begin{definition} \cite{soyarslanNoeth}  Let $X$ be a space and $\mathcal{A}\subseteq \tau^*(X)$. Then, we define and call \textbf{the purification of $\mathcal{A}$}, $pur( \mathcal{A})$, as follows:  $pur(\mathcal{A})=\{U\in \mathcal{A}: $ if $W \in \tau^*(U)$, then $W$ has a cellular family of cardinality $\omega\}$.
	
	Note that if $U\in pur(\mathcal{A})$, then $U$ has no isolated points.
	\begin{flushright}\kare	
		
	\end{flushright}
\end{definition}

\begin{remark}\label{r:g}\cite{soyarslanNoeth}
	Let $X$  be a  space. Suppose \2 has a winning strategy in $BM(X)$ and suppose we want to prove the existence of \2's 2-tactic (even k-tactic or k-Markov strategy). Let $U_i$ be a move of \1 in any arbitrary play of $BM(X)$. Then,  without loss of generality, we can suppose that $U_i\in pur(\tau^*(X))$. Otherwise  \2 has a 1-tactic (so, 2-tactic,  k-tactic or k-Markov) after the move $U_i$.

	In addition, it suffices to assume that $U_0\in pur(\tau^*(X))$ where $U_0$ is a first move of \1. This is because,  if $O$ is any move of \1 or \2 in this play after the move $U_0$, then $O\in pg(\tau^*(X))$.  Consequently,  if we assume $U_0\in pur(\tau^*(X))$, both \1 and \2  will continue playing from $pur(\tau^*(X))$.
\end{remark}

\begin{corollary}\label{makale}(Corollary 3.7 in \cite{soyarslanNoeth})
		Let $X$  be a  space. Suppose \2 has a winning strategy in $BM(X)$. For a $\pi$-base $\mathcal{B}$ of $X$ if $pur(\mathcal{B})=\emptyset$, then \2 has a 1-tactic (so, 2-tactic) in $BM(X)$. \begin{flushright}
		\kare
	\end{flushright}
\end{corollary}

\begin{definition}\label{r-d} Let $\mathcal{B}$ be a Noetherian $\pi$-base of a space $X$ with $pur(\mathcal{B})\neq \emptyset$ and $[pur(\mathcal{B})]$  a $\alpha \times \beta$ Noetherian table of $pur(\mathcal{B})$.  (We can give the same definiton where $[\mathcal{B}]$ is
	 an $\alpha \times \beta$ Noetherian table of $\mathcal{B}$.)

{\bf (a) Definition of $\widehat{U}$:} Let $U\in pur(\tau^*(X))$. Set $h(U)=\min \{o(B)\in \beta: B\in pur(\mathcal{B}), B\subseteq U\}$  and $s(U)= \min \{s\in \alpha: B^{h{(U)}}_s\in [pur(\mathcal{B})], B^{h{(U)}}_s\subseteq U\}$. 
Define  $\widehat{U}=B^{h{(U)}}_{s(U)}\in [pur(\mathcal{B})]$.

{\bf (b) Right-down version of $BM(X)$:} We modify the game $BM(X)$ as follows. Both \1 and \2 play  from $pur(\mathcal{B})$ and if $U$ is a move of the players \1 or \2, then $r(B)\geq r(U)$ and $o(B)\geq o(U)$ for any $B\in \{B\in pur(\mathcal{B}): B\subseteq U\}$.  We call the modified game  {\bf right-down version of $BM(X)$ on the Noetherian table $[pur(\mathcal{B})]$}. If there is no confusing,  shortly, we call it {\bf right-down version of $BM(X)$}.
	
\begin{flushright}
	\kare
\end{flushright} 

\end{definition}

\begin{remark}
Let $\mathcal{B}$ be a Noetherian $\pi$-base of a space $X$ with $pur(\mathcal{B})\neq \emptyset$. Note that 
 for any given $U\in pur(\mathcal{B})$, by using Definiton \ref{r-d} (a), we always  get $\widehat{U}\in pur(\mathcal{B})$.  Thus, \1 and \2 can use Definiton \ref{r-d} (a) for all their moves and so, they can play in the  right-down version of $BM(X)$. Even, the sequence $(\widehat{U}, \widehat{U}, \widehat{U}, \ldots)$ is a play  in the  right-down version of $BM(X)$.  Therefore, the game  right-down version of $BM(X)$ (on the Noetherian table $[\mathcal{B}]$ or $[pur(\mathcal{B})]$) exists if $pur(\mathcal{B})\neq \emptyset$.
 
 \begin{flushright}
 	\kare
 \end{flushright} 
\end{remark}

The following lemma can also be stated for the Noetherian table $[\mathcal{B}]$.

\begin{lemma} \label{Lemma1}
	
Let $\mathcal{B}$ be a $\pi$ base  with $pur(\mathcal{B})\neq \emptyset$ and $[pur(\mathcal{B})]$ a Noetherian table of a space $X$. 	Then, \2 has a winning strategy (k-tactic) in $BM(X)$  if and
	only if \2 has a winning strategy (k-tactic) in the right-down  version of $BM(X)$ on the Noetherian table $[pur(\mathcal{B})]$.
\end{lemma}

\begin{proof}From remark above, we know that right-down version of $BM(X)$ exists because $pur(\mathcal{B})\neq \emptyset.$

($\Rightarrow$) Let $\sigma$ be a winning strategy for \2 in $BM(X)$.  Now, we define a winning strategy $t$ for \2 in the right-down version of $BM(X)$. Define $t(U)=\widehat{\sigma(U)}$ for first round. For further rounds,
suppose $(U_0,V_0, U_1,\ldots,U_n)$ 
be an added play part  in 
the right down version of $BM(X)$. If $(U_0, \sigma(U_0), {U_1}, \ldots,\sigma(\ldots, U_{n-1}),{U_n})$ is an added 
 play part in $BM(X)$ and if  we can  set $V_0=\widehat{W_0}, \ldots, V_{n-1}=\widehat{W_{n-1}}$  where $W_0=\sigma(U_0),\ldots, W_{n-1}=\sigma(U_0, \sigma(U_0), {U_1},\ldots, U_{n-1})$, then we  define
   $$t(U_0,V_0, U_1,\ldots,U_n)=V_n$$ such that  $V_{n}=\widehat{W_{n}}$ where $W_{n}=\sigma(U_0, \sigma(U_0), {U_1},\ldots, U_{n})$.  If  not, we 
    define $$t(U_0,V_0, U_1,\ldots,U_n)=\widehat{U_n}.$$

(For $k$ tactic; let $\sigma$ be a k-tactic for \2 in   $BM(X)$. Let $(U_1, U_2\ldots,U_i)$ be a finite sequence of  last $i$ moves of \1 where $1\leq i\leq k$ in the right-down version of $BM(X)$.  Then set $V_i={\sigma( {U_1},\ldots,{U_i})}$ and define $t( U_1,\ldots,U_i)=\widehat{V_i}$.)

It is straightforward to see that $t$ is a winning strategy (k-tactic) for  \2 in the right-down version of $BM(X)$.

($\Leftarrow$)  Let $\sigma$ be a winning strategy  for \2 in the right-down version of  $BM(X)$. From Remark \ref{r:g},  if $U_0$ is a first move of \1 in any play in $BM(X)$, then we may suppose that $U_0\in pur(\tau^*(X))$. 

Suppose $(U_0,V_0, U_1,\ldots,U_n)$ be an added 
 play part  in $BM(X)$ with $U_0\in pur( \tau^*(X))$. If $(\widehat{U_0}, \sigma(\widehat{U_0}), \widehat{U_1},\sigma(\ldots\widehat{U_1}),\ldots,\widehat{U_n})$ is an added play part in the right-down version of $BM(X)$, then define $$t(U_0,V_0, U_1,\ldots,U_n)={\sigma(\widehat{U_0}, \sigma(\widehat{U_0}), \widehat{U_1},\ldots,\widehat{U_n})}.$$ If  not, define $t(U_0,V_0, U_1,\ldots,U_n)=U_n.$

(Let $\sigma$ be a k-tactic for \2 in the right-down version of  $BM(X)$. Let $(U_1,\ldots,U_i)$ be the  last $i$ moves of \1 where $1\leq i\leq k$ in $BM(X)$. If  $(\widehat{U_1},\ldots,\widehat{U_i})$  is the  last $i$ moves of \1 in the right down of $BM(X)$, then define $t(U_1,\ldots,U_i)=\sigma(\widehat{U_1},\ldots,\widehat{U_i})$. If not define $t(U_1,\ldots,U_i)=U_i.$)
 
It can be seen that $t$ is a winning strategy ($k$-tactic) in $BM(X)$.

\end{proof}

To simplify the notation used in the proof of the main theorem, we introduce the following definitions.

In the following definition, we let  $(A)^1=A$ and $(A)^n=\{(a_1,a_2,\dots ,a_n):\forall 1\leq k\leq n(a_k\in A)\}$. 

\begin{definition}\label{de}
	(1) Let $\gamma^m_l$ and $\theta^m_l$ be some ordinals for the integers $t,l$. Set $$\gamma^m_0\sim \theta^m_0=((\gamma^m_0,\theta^m_0))$$ and set  $$\gamma^m_0\sim \theta^m_k=((\gamma^m_0,\theta^m_0),(\gamma^m_1,\theta^m_1), (\gamma^m_2,\theta^m_2), \ldots, (\gamma^m_k,\theta^m_k))$$ where $k>0$ is an integer. 
	
	Now,  suppose $x=(\gamma^{0}_0\sim \theta^{0}_{k_0},\gamma^{1}_0\sim \theta^{1}_{k_1},\ldots, \gamma^{m}_0\sim \theta^{m}_{k_m})$ and $y=(\nu^{0}_0\sim \mu^{0}_{p_0},\nu^{1}_0\sim \mu^{1}_{p_1},\ldots, \nu^{n}_0\sim \mu^{n}_{p_n})$. Then define the product 
	 of $x$ and $y$ as follows $$x*y=(\gamma^{0}_0\sim \theta^{0}_{k_0},\ldots, \gamma^{m}_0\sim \theta^{m}_{k_n}, \nu^{0}_0\sim \mu^{0}_{p_0},\ldots, \nu^{n}_0\sim \mu^{n}_{p_n}).$$

(2)	Let $i$ and $j$ 
	be infinite  ordinals and $\mathcal{A}$ a family of sets such that $|\mathcal{A}|\geq \max\{|i|,|j|\}$.  What we mean with \textbf{a special numeration of $\mathcal{A}$ for $i$ and $j$},  denoted $\mathbb{S}_{\mathcal{A},i,j}$, is the following. 
	
	Let $$S=\{((\alpha,\beta)):(\alpha,\beta)\in i \times j\}\cup \big(\bigcup_{n\in \omega-\{0,1\}}(i \times j)^n\big).$$
	Then, define $$\mathbb{S}_{\mathcal{A},i,j}=\{(a):a\in S\}\cup(\bigcup_{n\in \omega-\{0,1\}} ({S})^n)$$ 
	
	Note that $|\mathbb{S}_{\mathcal{A},i,j}|\leq |\mathcal{A}|$. So, by fixing a one to one function $f$ from $\mathbb{S}_{\mathcal{A},i,j}$ to 
	$\mathcal{A}$, we can enumerate the subset  $f[\mathbb{S}_{\mathcal{A},i,j}]$ of $\mathcal{A}$ such that  $$f[\mathbb{S}_{\mathcal{A},i,j}]=\{O_x\in\mathcal{A}:x\in \mathbb{S}_{\mathcal{A},i,j}\text{ and } f(x)=O_x\}.$$

 Note that for any $x\in \mathbb{S}_{\mathcal{A},i,j}$, there are some ordinals $\gamma^t_l\in i$ and $\theta^t_l\in j$ such that $$x=(\gamma^{0}_0\sim \theta^{0}_{k_0},\gamma^{1}_0\sim \theta^{1}_{k_1},\ldots, \gamma^{m}_0\sim \theta^{m}_{k_m}).$$

Note that $x*y\in \mathbb{S}_{\mathcal{A},i,j}$ for  any $x,y \in \mathbb{S}_{\mathcal{A},i,j}$ and, so, for any $n\in \omega$, $x_0*x_1*\ldots * x_n\in \mathbb{S}_{\mathcal{A},i,j}$ where $x_0,x_1,\ldots x_n\in \mathbb{S}_{\mathcal{A},i,j}$.

Suppose $x=(\gamma^{0}_0\sim \theta^{0}_{k_0},\gamma^{1}_0\sim \theta^{1}_{k_1},\ldots, \gamma^{m}_0\sim \theta^{m}_{k_m})$, then we call $\gamma^{s}_0\sim \theta^{s}_{k_d}$ is {\bf a part of $x$} where $0\leq s\leq m$ and $0\leq d\leq n$. Suppose $\gamma^{s}_0\sim \theta^{s}_{k_d}$ be a part of $x$ and $\gamma^{s}_0\sim \theta^{s}_{k_d}=((\gamma^{s}_0,\theta^{s}_{0}),(\gamma^{s}_1, \theta^{s}_{1}), \ldots, (\gamma^{s}_{k_n}, \theta^{s}_{k_n}))$, then  we call  $(\gamma^{s}_{i}, \theta^{s}_{i})$ {\bf basic pair of $x$.}

	\begin{flushright}
		\kare
	\end{flushright} 
	
\end{definition}

The following definition can be given with Noetherian table $[\mathcal{B}]$ instead of $[pur(\mathcal{B})]$.

\begin{definition}\label{TAn}
Let $\mathcal{B}$ be a Noetherian $\pi$-base of a space $X$ with $pur(\mathcal{B})\neq \emptyset$ and $[pur(\mathcal{B})]$ be an $\alpha\times \beta$ Noetherian table of $pur(\mathcal{B})$. Suppose \2 has a winning strategy, $\sigma$, in the right-down version of  $BM(X)$.
 
 Let $U\in pur(\mathcal{B})$ with the condition that $r(B)\geq r(U), o(B)\geq o(U)$ for all $B\in \{B\in \mathcal{B}:B\subseteq U\}$ and let  
 and $x=(\gamma^{0}_0\sim \theta^{0}_{k_0},\gamma^{1}_0\sim \theta^{1}_{k_1},\ldots, \gamma^{m}_0\sim \theta^{m}_{k_m})$. If all the followings   can properly happen according mentioned rules and we get the following $W(x,U)$ and $x'$, then we call that  {\bf $x$ produces a $W(x,U)\in pur(\mathcal{B})$ with $\sigma$} and  we call  {\bf $x$ extends to $x'$ from $U$ with $\sigma$.}

It must be that $$B^{\theta^0_0}_{\gamma^0_0} \ldots, B^{\theta^0_{k_0}}_{\gamma^0_{k_0}},	B^{\theta^1_0}_{\gamma^1_0}, \ldots, B^{\theta^1_{k_1}}_{\gamma^1_{k_1}},\ldots,
B^{\theta^m_0}_{\gamma^m_0},\ldots, B^{\theta^m_{k_m}}_{\gamma^m_{k_m}}\in [pur(\mathcal{B})].$$

Then, $$(B^{\theta^0_0}_{\gamma^0_0}, \sigma(B^{\theta^0_0}_{\gamma^0_0}), \ldots, B^{\theta^0_{k_0}}_{\gamma^0_{k_0}}, \sigma(\ldots, B^{\theta^0_{k_0}}_{\gamma^0_{k_0}}), U, \sigma (\ldots,\sigma(\ldots,B^{\theta^0_{k_0}}_{\gamma^0_{k_0}}), U))$$ 
must be a  play part in the right-down version of  $BM(X)$.  After that, when $\sigma (\ldots,\sigma(\ldots,B^{\theta^0_{k_0}}_{\gamma^0_{k_0}}), U)=W_0$,

$$(B^{\theta^1_0}_{\gamma^1_0}, \sigma(B^{\theta^1_1}_{\gamma^1_0}), \ldots, B^{\theta^1_{k_1}}_{\gamma^1_{k_1}}, \sigma(\ldots, B^{\theta^1_{k_1}}_{\gamma^1_{k_1}}),W_0, \sigma(\dots\sigma(\ldots,B^{\theta^1_{k_1}}_{\gamma^1_{k_1}}),W_0))$$
must be a play part in the right-down version of  $BM(X)$.  After that, when $\sigma(\dots\sigma(\ldots,B^{\theta^1_{k_1}}_{\gamma^1_{k_1}}),W_0)=W_1$,

$\ldots$

$\ldots$
	$$(B^{\theta^m_0}_{\gamma^m_0}, \sigma(B^{\theta^m_0}_{\gamma^m_0}), \ldots, B^{\theta^m_{k_m}}_{\gamma^m_{k_m}}, \sigma(\ldots, B^{\theta^m_{k_m}}_{\gamma^m_{k_m}}),W_{m-1}, \sigma(\ldots\sigma(\ldots,B^{\theta^m_{k_m}}_{\gamma^m_{k_m}}),W_{m-1}))$$ must be a play part in the right-down version of  $BM(X)$.  
	
	Lastly,  we set	 $$\sigma(\ldots\sigma(\ldots,B^{\theta^m_{k_m}}_{\gamma^m_{k_m}}),W_{m-1})=W_m=W(x,U).$$ 
Now, say $(\gamma^{0}_{k_0+1},\theta^{0}_{k_0+1})=(r(U),o(U))$ and $(\gamma^{1}_{k_1+1},\theta^{1}_{k_1+1})=(r(W_0),o(W_0))$,  $(\gamma^{2}_{k_2+1},\theta^{2}_{k_2+1})=(r(W_1),o(W_1))$, $\ldots$, $(\gamma^{m}_{k_m+1},\theta^{m}_{k_m+1})=(r(W_{m-1}),o(W_{m-1})).$ After all that, we can set $$x'=(\gamma^{0}_0\sim \theta^{0}_{k_0+1},\gamma^{1}_0\sim \theta^{1}_{k_1+1},\ldots, \gamma^{m}_0\sim \theta^{m}_{k_m+1}).$$

	\begin{flushright}
	\kare
\end{flushright} 

\end{definition}

\vspace{2mm}\textbf{The Main Theorem}\vspace{2mm}

The following theorem, is also true  for Noetherin table $[\mathcal{B}]$ instead of $[pur(\mathcal{B})]$.

\begin{theorem}\label{Teorem winning ten 2-tactige ana sonuc open-finite} 
	Let $X$ be a space and  $\mathcal{B}$ a $\pi$-base of $X$ having the following ($\circledast$)-condition. If \2 has a winning strategy in $BM(X)$, then \2 has a  2-tactic in $BM(X)$.

	($\circledast$) $pur(\mathcal{B})=\emptyset$  or 
	 $pur(\mathcal{B})$ is   Noetherian and has a 
	Noetherian table $[pur(\mathcal{B})]$ such that
	for any $B\in pur(\mathcal{B})$ if  $o(K)\geq o(B)$ for all $K\in \{K\in pur(\mathcal{B}):K\subseteq B\}$, then
		there exists an open-finite family $\mathcal{A}_{B}\subseteq \tau^*(B)$ such that 
	$|\mathcal{A}_{B}|\geq \max\{ |r(B)|,M_B\}$ where Noetherian union of $pur(\mathcal{B})$ is  $\bigcup_{\theta<\alpha}\mathcal{B}_\theta$ and $M_B=\sup\{|\{V\in  \mathcal{B}_\theta:B\subseteq V, o(V)\leq o(B)\}|:\theta\leq r(B)\}$.

\end{theorem}

\begin{proof}

 If $pur(\mathcal{B})=\emptyset$, from Corollary \ref{makale}, the proof is completed.
	Suppose $pur(\mathcal{B})\neq \emptyset$. Let  $[pur(\mathcal{B})]$ be an  $\alpha\times\beta$ Noetherian table and say Noetherian union of $pur(\mathcal{B})$ is	$\bigcup_{i<\alpha}\mathcal{B}_i$.
	Say any $i$-th row of  $[pur(\mathcal{B})]$ is $(B^0_i,B^1_i, B^2_i, \ldots)$.

To get the  following two definitions take  and fix any  $B\in pur(\mathcal{B})$.
For $\widehat{B}$ see Definiton \ref{r-d} (a).

(1-Definitions of $\mathcal{C}_{B,\theta}, m_{B,\theta}$) Take any  $\theta \leq r(B)$.  Say $|\{V\in  \mathcal{B}_\theta:B\subseteq V, o(V)\leq o(B)\}|=m_{B,\theta}$ and  $\mathcal{C}_{B,\theta}=\{U\in  \mathcal{B}_\theta:B\subseteq U, o(U)\leq o(B)\}$. Then we can rewrite  and fix $\mathcal{C}_{B,\theta}=\{U_l:l \in m_{B,\theta}\}$ with $U_l\neq U_t$ if $l\neq t$.   
(So,
 $M_{B}=\sup\{m_{B,\theta}:\theta \leq r(B)\}$.)

(2-Definition of $\mathcal{O}_{{B}}$) Because $B\in pur(\mathcal{B})$, ${B}$ has a cellular family of cardinality $\geq\omega$. Now, suppose  $o(K)\geq o(B)$ for all $K\in \{K\in pur(\mathcal{B}):K\subseteq B\}$, then by using the  hypothesis, we can fix an open finite family $\mathcal{A}_{B}\subseteq \tau^*(B)$ 
such that $|\mathcal{A}_{{B}}|\geq \max\{ \omega,|r({B})|, M_{{B}}\}$ (Because  $B\in pur(\mathcal{B})$, there exists a disjoint (so, open finite) family  $\mathcal{A}\subseteq \tau^*(B)$ with $|\mathcal{A}|\geq \omega$. Thus, even if $|r({B})|, M_{{B}}<\omega$, we always   assume that $|\mathcal{A}_{{B}}|\geq  \omega$. ). 
Now, set $i_B= \max\{ \omega,r({B}), M_{{B}}\}+1$. ($+1$ means ordinal succoser) 
Let $\mathbb{S}_{\mathcal{A}_{B}, i_B,i_B}$ be a special numeration of $\mathcal{A}_{B}$. 
And let

 $$\mathcal{O}_{B}=\{O_{\bar{x}}\in\mathcal{A}_{B}:\bar{x}\in  \mathbb{S}_{\mathcal{A}_{B},i_B,i_B}\}.$$ 
	 Because of definition of $\mathbb{S}_{\mathcal{A}_{B},i_B,i_B}$,  $O_{\bar{x}}\neq O_{\bar{y}}$ if ${\bar{x}}\neq {\bar{y}}$.

For any $B\in pur(\mathcal{B})$, let us put a well order on $\mathbb{S}_{\mathcal{A}_{B},i_B,i_B}$.

Let us be given $U,V\in pur(\mathcal{B})$ with the condition that $r(U)\leq r(V)$ and $o(U)\leq o(V)$ and $U\supseteq V$. So, $U\in \mathcal{C}_{V,r(U)}=\{U_l:l \in m_{V,r(U)}\}$. Thus, there is exactly one $l_0\in m_{V,r(U)}$ such that $U=U_{l_0}.$
 Now, set $$o_V(U)=l_0.$$ Note that suppose we know $V$, then we have 
 $(r(U),o(U))$ if and only if we have 
 $(r(U),o_V(U))$. This fact will be used throughout the proof without mentioning it.
 
 In this proof, any basic pair of any part of $\bar{x}\in  \mathbb{S}_{\mathcal{A}_{V},i_V,i_V}$   will be in the form of  $(r(U),o_V(U))$ for some $U,V\in [pur(\mathcal{B})]$. By exchanging    each  basic pair $(r(U),o_V(U))$ for $(r(U),o(U))$, we define an $x$ from $\bar{x}$. More formaly, suppose $\bar{x}=\gamma^{0}_0\sim \bar{\theta}^{0}_{k_0},\gamma^{1}_0\sim \bar{\theta}^{1}_{k_1},\ldots, \gamma^{m}_0\sim \bar{\theta}^{m}_{k_m}$. So, we can say that the first part of $\bar{x}$ is  $\gamma^{0}_0\sim \bar{\theta}^{0}_{k_0}=((r(U_0), o_V(U_0)),(r(U_1), o_V(U_1)),\ldots,(r(U_{k_0}), o_V(U_{k_0})))$. 
  Thus,  first part of $x$ is ${\gamma^{0}_0\sim {\theta^{0}_{k_0}}}=((r(U_0), o(U_0)),(r(U_1), o(U_1)),\ldots,(r(U_{k_0}), o(U_{k_0})))$. After applying similar thing for all other parts,
  we get $x=\gamma^{0}_0\sim \theta^{0}_{k_0},\gamma^{1}_0\sim \theta^{1}_{k_1},\ldots, \gamma^{m}_0\sim \theta^{m}_{k_m}$.

 Because \2 has a winning strategy in $BM(X)$, from Lemma \ref{Lemma1}, \2 has a winning strategy  $\sigma$ in the right-down version of  $BM(X)$.

Now, we define a 2-tactic $t$ for \2 in the right-down version of $BM(X)$ as follows. 

For any given $U\in pur(\mathcal{B})$,  define

 $$t(U)=\sigma(U).$$ 
(Note that because the plays will be in the right-down version of $BM(X)$, $r(B)\geq r(U)$ and $o(B)\geq o(U)$ for any $B\in \{B\in pur(\mathcal{B}): B\subseteq U\}$. And the same applies to the following U and V.) 
 
 To define $t(U,V)$ in the right-down version of $BM(X)$, let us be given $U,V\in pur(\mathcal{B})$.

Say $\mathcal{F}_{U,V}=\{\bar{x}\in \mathbb{S}_{\mathcal{A}_{U}, i_U,i_U}:x$ produces a $W(x,\widehat{O_{\bar{x}}})$   with $\sigma$ where  $O_{\bar{x}}\in \mathcal{O}_U$ and 
 $V\subseteq W(x,\widehat{O_{\bar{x}}})\}$. 
 Now, $\{O_{\bar{x}}\in \mathcal{O}_U: V\subseteq O_{\bar{x}},  \bar{x}\in \mathcal{F}_{U,V} \}$ is finite because   $\mathcal{O}_U$ is a subset of open finite family $\mathcal{A}_{U}$. Thus,  $\mathcal{F}_{U,V}$ is finite.  

\textbf{Case 1.} If $\mathcal{F}_{U,V}= \emptyset$, then set $\bar{x}_{U,V}=\big(((r(U),o_V (U)))\big)$. Because $U\in C_{V,r(U)}$, $o_V (U)\leq m_{V,r(U)}\leq M_V<i_V$. So, $o_V(U)<i(V)$. In addition, $r(U)\le i_V$ because $U\supseteq V$ and $i_V>r(V)\geq r(U)$. Thus, $\bar{x}_{U, V}\in \mathbb{S}_{\mathcal{A}_V,i_V,i_V}$. So, there is $\widehat{O}_{\bar{x}_{U,V}}$ and we can get $x_{U,V}$ from $\bar{x}_{U, V}$. Note that if $(U,\sigma(U), \widehat{O}_{\bar{x}_{U,V}})$ is an added play part in the right-down version of $BM(X)$, then $x_{U,V}$ produces $W(x_{U, V},\widehat{O_{\bar{x}_{U,V}}})$ with $\sigma$. Then, 
 define $$t(U,V)=W(x_{U, V},\widehat{O_{\bar{x}_{U,V}}}).$$ 
 If not, i.e., $(U,\sigma(U), \widehat{O}_{\bar{x}_{U,V}})$ is not an added play part
  in the right-down version of $BM(X)$, then define $t(U,V)=V.$

\textbf{Case 2.} If $\mathcal{F}_{U,V}\neq \emptyset$, by using the well order on  $\mathbb{S}_{\mathcal{A}_{U}, i_U, i_U}$, we can enumerate $\mathcal{F}_{U,V}=\{\bar{x}_0,\bar{x}_1,\ldots,\bar{x}_n\}$ such that $\bar{x}_0<\bar{x}_1<\ldots<\bar{x}_n.$
  Thus, for $0\leq i\leq n$, we get $x_i$ from $\bar{x}_i$ 
and  $x_i$ extends to $x'_i$ from $\widehat{O_{\bar{x_i}}}$   with $\sigma$ where  $O_{\bar{x}_i}\in \mathcal{O}_U$.  
Now, if 
$(U,\sigma(U),V)$ 
is an added play part in the right-down version of $BM(X)$, then set  $$x_{U,V}=(((r(U),o(U))))*x'_0*x'_1*\ldots*x'_n.$$ If not, then set
$$x_{U,V}=x'_0*x'_1*\ldots*x'_n.$$
Note that we can get $\bar{x}_{U,V}$ 
from $x_{U,V}.$

\textbf{Claim.} $\bar{x}_{U,V}\in \mathbb{S}_{\mathcal{A}_{V}, i_V,i_V}$ and $x_{U,V}$ produces a 
$W(x_{U,V},\widehat{O_{\bar{x}_{U,V}}})\in pur(\mathcal{B})$ with $\sigma$ where $O_{\bar{x}_{U,V}}\in \mathcal{O}_V\subseteq \mathcal{A}_{V}$. 

\textbf{Proof of the Claim.}
Now, let $(r(H),o(H))$ be a basic part of a ${x'}_i$ and consider $(r(H),o_V (H))$. Then, by using Definition \ref{de} and \ref{TAn},  it is straightforward to see that $r(H)<i_V$ and $o_V(H)<i_V$. Thus, $\bar{x'}_i\in \mathbb{S}_{\mathcal{A}_{V}, i_V,i_V}$. Therefore,   $\bar{x}_{U,V}\in \mathbb{S}_{\mathcal{A}_{V}, i_V,i_V}$.
 Because $\widehat{O_{\bar{x}_{U,V}}}\subseteq V$, from   Definition \ref{de} and \ref{TAn},  $x_{U,V}$ produces a 
$W(x_{U,V},\widehat{O_{\bar{x}_{U,V}}})\in pur(\mathcal{B})$ with $\sigma$ where $O_{\bar{x}_{U,V}}\in \mathcal{O}_V\subseteq \mathcal{A}_{V}$. So, proof of the claim is completed.

Thus, we define $$t(U,V)=W(x_{U,V},\widehat{O_{\bar{x}_{U,V}}}).$$

To see $t$ is a 2-tactic,  we will get an  arbitrary  play  

\begin{equation}\label{2}
U_0,t(U_0),U_1,t(U_0,U_1),U_2,t(U_1,U_2),U_3,\ldots	
\end{equation}in the right-down version of $BM(X)$ where \2 moves  by using  $t$.
After that, we will see that $\bigcap_{i\in \omega}U_i\neq \emptyset$. To see it, as we construct Play (\ref{2}), by using induction we will define a play $(H_0, \sigma(H_0),H_1, \sigma(\ldots,H_1),H_2,\ldots)$ in the right-down version of $BM(X)$ such that $U_0\supseteq H_0\supseteq \sigma(H_0)\supseteq t(U_0) \supseteq U_1\supseteq H_1\supseteq \sigma(\ldots,H_1)\supseteq t(U_{0},U_1)\supseteq\dots$ .

For the first round, let \1 moves $U_0$. Say $H_0=U_0$, then from the definition of $t$, $t(U_0)=\sigma(H_0)$. So, we get a play part $(H_0,\sigma(H_0))$ and $U_0\supseteq H_0\supseteq \sigma(H_0)\supseteq t(U_0).$

For the second round, let \1 moves $U_1$. \2 sees $U_0$ and $U_1$. If  $\mathcal{F}_{U_0,U_1}= \emptyset$, then from the definition of $t$, we get $x_{U_0, U_1}$ from  $\bar{x}_{U_0,U_1}=((r(U_0),o_{U_1}(U_0)))$. Note that   $x_{U_0, U_1}$ produces $W(x_{U_0, U_1},\widehat{O_{\bar{x}_{U_0, U_1}}})=\sigma(U_0,\sigma(U_0),\widehat{O_{\bar{x}_{U_0, U_1}}})$. From the definition of $t$,
$t(U_0,U_1)=W(x_{U_0, U_1},\widehat{O_{\bar{x}_{U_0, U_1}}})$. 
If $\mathcal{F}_{U_0,U_1}\neq \emptyset$, suppose $\mathcal{F}_{U_0,U_1}=\{\bar{x}_0,\bar{x}_1,\ldots,\bar{x}_n\}$ such that $\bar{x}_0<\bar{x}_1<\ldots<\bar{x}_n$.
Because $(U_0,\sigma(U_0),U_1)$ 
is an added
 play part in the right-down version of $BM(X)$, then $x_{U_0,U_1}=((r(U_0),o(U_0)))*x'_0*x'_1*\ldots*x'_n.$ Because  this  $x_{U_0, U_1}$ produces this $W(x_{U_0, U_1},\widehat{O_{\bar{x}_{U_0, U_1}}})$,  from the definition of $t$, \2 makes the move 
$t(U_0,U_1)=W(x_{U_0, U_1},\widehat{O_{\bar{x}_{U_0, U_1}}})$. Thus, for both  cases, i.e., either $\mathcal{F}_{U_0,U_1}= \emptyset$ or $\mathcal{F}_{U_0,U_1}\neq \emptyset$, we  set $H_1=\widehat{O_{\bar{x}_{U_0, U_1}}}$. Thus, by using definition of $t$ and Definition \ref{TAn}, $\sigma(H_0,\sigma(H_0),H_1)=\sigma(U_0,\sigma(U_0), \widehat{O_{\bar{x}_{U_0, U_1}}})\supseteq W(x_{U_0, U_1},\widehat{O_{\bar{x}_{U_0, U_1}}}) =t(U_0,U_1).$
Therefore, we get $U_0\supseteq H_0\supseteq \sigma(H_0)\supseteq t(U_0)\supseteq U_1\supseteq H_1\supseteq \sigma(\ldots,H_1)\supseteq t(U_{0},U_1).$

Note that  either $\mathcal{F}_{U_0,U_1}= \emptyset$ or $\mathcal{F}_{U_0,U_1}\neq \emptyset$, $((r(U_0),o(U_0)))$ is a 
part of $x_{U_0,U_1}$ and $H_0=B^{o(U_0)}_{r(U_0)}\in [pur(\mathcal{B})]$. And $(r(U_0),o(U_0)), (r(\widehat{O_{\bar{x}_{U_0, U_1}}}),o(\widehat{O_{\bar{x}_{U_0, U_1}}}))$ is also a part of $x'_{U_0,U_1}$ with $H_1=\widehat{O_{\bar{x}_{U_0, U_1}}}$. So, we have the hypothesis of the following induction for $n=1$.

We use induction. 

For an $n\geq1$, as  an induction hypothesis suppose that for (n+1)-th round \1 makes   $U_n$ and we have the following.

 ($1_n$) \2 makes the move  $t(U_{n-1},U_n) =W(x_{U_{n-1},U_n},\widehat{O_{\bar{x}_{U_{n-1}, U_n}}})$ where  ${O_{\bar{x}_{U_{n-1}, U_n}}}\in \mathcal{O}_{U_n}$ and ${\bar{x}_{U_{n-1}, U_n}}\in \mathbb{S}_{\mathcal{A}_{U_n},i_{U_n},i_{U_n}}$.

($2_n$) There is a play part $$(H_0,\sigma(H_0),H_1,\sigma(H_0,\sigma(H_0),H_1),H_2,\ldots, H_n,\sigma(\ldots,H_n))$$ in the right-down version of $BM(X)$ such that $$U_0\supseteq H_0\supseteq \sigma(H_0)\supseteq t(U_0)\supseteq\ldots \supseteq U_n\supseteq H_n\supseteq \sigma(\ldots,H_n)\supseteq t(U_{n-1},U_n).$$

($3_n$) There is a part ${\gamma_0}\sim {\theta_{n-1}}$ of $x_{U_{n-1},U_n}$ such that  $H_0=B^{\theta_0}_{\gamma_0},H_1=B^{\theta_1}_{\gamma_1}, \ldots, H_{n-1}=B^{\theta_{n-1}}_{\gamma_{n-1}}.$ As in Definition \ref{TAn}, while extending $x_{U_{n-1},U_n}$ to $x'_{U_{n-1},U_n}$ we get the part ${\gamma_0}\sim {\theta_{n}}$ of $x'_{U_{n-1},U_n}$ from ${\gamma_0}\sim {\theta_{n-1}}$ 
by adding $(\gamma_n, \theta_{n})$ to the end of ${\gamma_0}\sim {\theta_{n-1}}$. 
So, we can set   $H_n=B^{\theta_n}_{\gamma_{n}}\in [pur(\mathcal{B})]$.  

 (Note that if $n=1$, then  
$\gamma_0\sim \theta_{n-1}={\gamma^0_0}\sim {\theta^0_{n-1}}$. 
In other words, ${\gamma_0}\sim {\theta_{n-1}}$ 
is the first part of $x_{U_{n-1},U_n}=x_{U_{0},U_1}$. If $n>1$, then  $\gamma_0\sim \theta_{n-1}$ does not have to be first part of $x_{U_{n-1},U_n}$. So, we just write 
${\gamma_0}\sim {\theta_{n-1}}$ instead of ${\gamma^k_0}\sim {\theta^k_{n-1}}$.)

For $n+1$; suppose that for (n+2)-th round, \1 makes   $U_{n+1}$ such that $U_{n+1}\subseteq t(U_{n-1},U_n)$.

($1_{n+1}$) $\bar{x}_{U_{n-1},U_n}\in \mathcal{F}_{U_n,U_{n+1}}$, because of  induction hypothesis ($1_n$) and definition of $\mathcal{F}_{U_n,U_{n+1}}$. So, $\mathcal{F}_{U_n,U_{n+1}}\neq \emptyset$.  Thus,  by applying the   process of Case 2 and by using  Claim, we get  
 $\bar{x}_{U_{n},U_{n+1}}\in \mathbb{S}_{\mathcal{A}_{U_{n+1}}, i_{U_{n+1}}, i_{U_{n+1}}}$ and $x_{U_{n},U_{n+1}}$
  produces a 
 $W(x_{U_{n},U_{n+1}},\widehat{O_{\bar{x}_{U_{n},U_{n+1}}}})\in pur(\mathcal{B})$ 
 	with $\sigma$ where $O_{\bar{x}_{{U_{n},U_{n+1}}}}\in \mathcal{O}_{U_{n+1}}\subseteq \mathcal{A}_{U_{n+1}}.$ Therefore,
 \2 makes the move $t(U_{n},U_{n+1})=W(x_{U_{n},U_{n+1}},\widehat{O_{\bar{x}_{U_{n},U_{n+1}}}}).$
 Hence, 
we have ($1_{n+1}$). 

($2_{n+1}$ and $3_{n+1}$)  From   the paragraph   above  we know $\bar{x}_{U_{n-1},U_n}\in \mathcal{F}_{U_n,U_{n+1}}$ and from the process of Case 2, we know that $x_{U_{n},U_{n+1}}=\ldots *x'_{U_{n-1},U_n}*\ldots$ . From Definition \ref{TAn} and ($3_n$), the part  ${\gamma_0}\sim {\theta_{n-1}}$ of ${x}_{U_{n-1},U_n}$ will be ${\gamma_0}\sim {\theta_{n}}$ and ${\gamma_0}\sim {\theta_{n}}$ will be a part of $x'_{U_{n-1},U_n}$. So, ${\gamma_0}\sim {\theta_{n}}$ will be a part of  $x_{U_{n},U_{n+1}}$. 
Because $x_{U_{n},U_{n+1}}$ produces a $W(x_{U_{n},U_{n+1}},\widehat{O_{\bar{x}_{U_{n},U_{n+1}}}})$, As in Definition \ref{TAn}, while the process of producing $x_{U_{n},U_{n+1}}$ will extends to $x'_{U_{n},U_{n+1}}$. Thus, we will get the part ${\gamma_0}\sim {\theta_{n+1}}$ of $x'_{U_{n},U_{n+1}}$  from ${\gamma_0}\sim {\theta_{n}}$ 
by adding $(\gamma_{n+1}, \theta_{n+1})$ to the end of ${\gamma_0}\sim {\theta_{n}}$. 
Now, we can set $H_{n+1}=B^{\theta_{n+1}}_{\gamma_{n+1}}$. 
Thus, with $(2_n)$ and $(3_n)$ and  by using Definition \ref{TAn},  we get 
($2_{n+1}$), i.e., there is a play part $$(H_0,\sigma(H_0),H_1,\sigma(H_0,\sigma(H_0),H_1),H_2,\ldots, H_{n+1},\sigma(\ldots,H_{n+1}))$$ in the right-down version of $BM(X)$ such that $U_0\supseteq H_0\supseteq \sigma(H_0)\supseteq t(U_0)\supseteq\ldots \supseteq U_n\supseteq H_n\supseteq \sigma(\ldots,H_n)\supseteq t(U_{n-1},U_n)\supseteq U_{n+1}\supseteq H_{n+1}\supseteq \sigma(\ldots,H_{n+1})\supseteq t(U_{n},U_{n+1}).$
 And we  also get ($3_{n+1}$).

Consequently, when Play (\ref{2})  ends, from the induction, we can say that there exists a play 
$$H_0,\sigma(H_0),H_1,\sigma(H_0,\sigma(H_0),H_1),H_2,\sigma(\ldots,H_{2}), \ldots$$in the right-down version of $BM(X)$. Because $\sigma$ is winning strategy of \2,  $\bigcap_{i\in \omega}H_i\neq \emptyset$. Also, from the induction, this play entwined with Play (\ref{2}) such that 

$$U_0\supseteq H_0\supseteq \sigma(H_0)\supseteq t(U_0)\supseteq\ldots \supseteq U_n\supseteq H_n\supseteq \sigma(\ldots,H_n)\supseteq t(U_{n-1},U_n)\supseteq \ldots .$$

Hence, $\emptyset\neq \bigcap_{i\in \omega}H_i=\bigcap_{i\in \omega}U_i$. Therefore, $t$ is a 2-tactic for \2 in the right-down version of $BM(X)$. 

Hence, from Lemma \ref{Lemma1}, \2 has a 2-tactic in $BM(X)$.

  \end{proof}

 Clearly, Theorem \ref{Teorem winning ten 2-tactige ana sonuc open-finite} includes Theorem 3.8 in \cite{soyarslanNoeth}. Since Theorem 3.8 in \cite{soyarslanNoeth} contains Galvin's theorem (Corollary 3.10 in \cite{soyarslanNoeth}), our main result, Theorem \ref{Teorem winning ten 2-tactige ana sonuc open-finite}, also encompasses Galvin's theorem.

\bibliographystyle{plain}

\end{document}